\documentclass[12pt]{amsart}
\usepackage[margin = 1in]{geometry}
\allowdisplaybreaks
\usepackage{tikz}
\usepackage{enumitem}
\usepackage{amssymb}
\newcommand{\1}{1\!\!\!\!\!\;{\rm I}}
\newcommand{\mbR}{{\mathbb R}}

\newcommand{\cC}{{\mathcal C}}
\newcommand{\cD}{{\mathcal D}}
\newcommand{\cS}{{\mathcal S}}
\newcommand{\cG}{{\mathcal G}}

\newcommand{\ba}{\begin{aligned}}
\newcommand{\ea}{\end{aligned}}

\newcommand{\wt}{\widetilde}

\newcommand{\ve}{\varepsilon}

\newcommand{\Pb}{\mathrm{P}} 
\newcommand{\E}{\mathrm{E}} 
\newcommand{\be}{\begin{equation}}
\newcommand{\bel}{\begin{equation}\label}
\newcommand{\ee}{\end{equation}}

\begin{document}
\title[Boundary Approximation for Sticky Jump-Reflection]{Boundary Approximation for Sticky Jump-Reflected Processes on the Half-Line}
\author{Andrey Pilipenko}
\address{Institute of Mathematics, National Academy of Sciences of Ukraine, Kyiv; Igor Sikorsky Kyiv Polytechnic Institute}
\email{pilipenko.ay@gmail.com}
\author{Andrey Sarantsev}
\address{Department of Mathematics and Statistics, University of Nevada, Reno}
\email{asarantsev@unr.edu}
\keywords{Reflecting map; Skorokhod map; Sticky processes; Reflecting barriers; Invariance principle; Feller-Wentzell boundary conditions}
\subjclass[2020]{Primary: 60F17, 60J50; Secondary: 60J55; 60K25}

\begin{abstract}
The Skorokhod reflection was used in 1961 to create a reflected diffusion 
on the half-line. Later, it was used for processes with jumps such as reflected L\'evy processes. Like a Brownian motion, which is a weak limit of random walks, reflected processes on the half-line serve as weak limits of random walks with switching regimes at zero: one regime away from zero, the other around zero. 
 In this article, we develop a general theory of this regime change and prove convergence to a function with generalized reflection. Our results are deterministic and can be applied to a wide class of stochastic processes.
  Applications include storage processes, heavy traffic limits,  
      diffusion on a half-line with  a combination of 
    continuous reflection, jump exit, and a delay at 0.

\end{abstract}
\maketitle
\thispagestyle{empty}

\newtheorem{thm}{Theorem}
\newtheorem{lemma}{Lemma}
\newtheorem{corl}{Corollary}

\theoremstyle{definition}
\newtheorem{defn}{Definition}
\newtheorem{remk}{Remark}
\newtheorem{expl}{Example}

\section{Introduction}

In a classic M/M/1 queue, customers arrive at rate $\lambda$, are served at rate $\mu$, with inter-arrival and service independent exponential times. 
The number of customers in the queue  is a birth-and-death discrete-valued Markov process on $\{0, 1, 2, \ldots\}$ with birth rate $\lambda$ and death rate $\mu.$  Now, consider a sequence of queues as $\lambda \uparrow \mu$ ({\it heavy traffic limit}). With correct normalization, we can prove that a scaling limit of   these queues is  a reflected Brownian motion, possibly with drift. This is an $[0, \infty)$-valued process which behaves as a classic Brownian motion away from zero, and is reflected instantaneously (according to the Skorokhod reflection) at zero. We refer the reader to the classic book \cite{Whitt}. The method of proof is {\it continuous mapping:} We have a Skorokhod mapping which takes a function and makes a reflected version out of it. Applied to a random walk, this mapping gives us an M/M/1 queue. Applied to a Brownian motion, this gives us reflected Brownian motion. It is well known that a properly re-scaled random walk weakly converges to a Brownian motion. Since the Skorokhod mapping is continuous in 
the corresponding functional space, heavy traffic limit of M/M/1 queues is indeed a reflected Brownian motion, see also   \cite{Asmussen, Kushner, Whitt}.

However, Skorokhod reflection is far from the only reflection model for a Brownian motion, and, more generally, for diffusion or L\'evy processes. In addition to the Skorokhod reflection  the following boundary behavior are possible: 

\begin{itemize}
    \item Absorption: Hitting zero and stopping there
    \item Delayed/sticky reflection: Spending more time at zero than with classic Skorokhod reflection. Time spent at zero has positive Lebesgue measure.
    \item Jump reflection: Jumping out of zero upward, instead of reflecting continuously. 
\end{itemize}

Classification of reflection modes were done by Feller and Wentzell in \cite{Feller, Wentzell}. The construction was done by using  resolvents,   semigroups,  their generators, and boundary conditions. Probabilistic approach was originated by Ito and McKean \cite{Ito}, where the excursion theory was used. The work on continuity principle for these more general reflections is a hard problem. Delay models with Skorokhod's reflection are easier to 
study rather than jump-exit from the boundary, see for example \cite{HarrisonLemoine, Yamada1994ReflDelay}. This can be explained by nice properties of the Skorokhod reflecting map and the fact that a delay at the boundary is regulated by a local time at the boundary, which can be associated with a regulator boundary term of the Skorokhod problem. Construction of processes  with jump-exit  from the boundary is not a trivial task. For example,  the corresponding processes were constructed via an elegant usage of excursion theory \cite{Watanabe1975Wentzell_boundary},  see also an approach based on
the mixture of stochastic differential equations  together with martingale problem methods, and resolvent analysis \cite{Anulova, MikulyaviciusExist, MikuljaviciusUniq}. Jump-reflected Brownian motion    was constructed in \cite[Theorem 3.11]{Excursion} as a function of a Wiener process  and a subordinator, whose L\'evy measure coincides with an `intensity of exit from 0' of the jump-reflected Brownian motion.

In this article, we consider a general deterministic setting.   Instead of a queue, we consider a (deterministic) {\it switch problem:} It behaves differently at the boundary than away from zero.  For a limit function, we create a 
reflection mapping on the space of RCLL functions, which generalizes the classic Skorokhod mapping mentioned above and includes delays and jumps.   Then  convergence results  for various stochastic  models
will be received automatically as an application of continuous map theorem.
As an application of the general deterministic result  on convergence, we will be able to handle out limit theorems for storage processes, perturbed reflected random walks, spectrally positive L\'evy processes having jump-type reflection at 0. 

Let us concisely explain this {\it switch problem}. 
Definition~\ref{defn:switch} below in Section 3 provides a rigorous statement.
We have an input function $x$
driving the system in the regular regime, and a regulator $F$ for the boundary regime. The output function $y$ evolves together with $x$ (that is, $x$ and $y$ have the same increments) as long as $y$ does not get below $-\delta$. After that, we stop $x$ and restart $y$ as $F$, until $y$ gets above 0. Then we start $y$ again as $x$ from the place where we stopped $x$ previously, until $y$ gets below $-\delta$. Then restart $y$ again as $F$ from the place where we stopped $F$ last time, etc. 

Informally, this can be illustrated as follows: There is an Internet shop, which stores goods and gradually sells them to customers. When the inventory reaches zero (or even  some negative  level, i.e., there are more orders than goods), the shop switches to a critical regime. It is still accepting orders, but starts to order in larger batches. When the shop again has enough on hand, they stop to order in larger batches and 
continue the usual functioning.
Our goal in this article is to show that the solution to the switch problem converges to the solution to the generalized Skorokhod problem as the threshold $\delta$ between
the critical and usual regimes converges to 0. We stress that we show this in the general case for deterministic functions. 
This creates the framework for proving such weak convergence results for regulated stochastic processes.

 Moreover, we will obtain 
continuous dependence on controls in usual and critical regimes in the following sense.
We take two sequences of functions: $x_n \to x_0$ and $F_n \to F_0$. Next, we take a sequence of non-negative numbers $\delta_n \to 0$ and another sequence $\rho_n \to \rho \in [0, \infty]$ of time normalizing constants. For each $n$, we solve the switch problem with input function $x_n$, threshold $-\delta_n$, and regulator $F_n$ that is slowed down by the factor of $\rho_n$.

If $F_0$ is strictly increasing, and under some additional technical assumptions, then the solution to this switch problem converges to the reflected process as in~\eqref{eq:general-Skorohod} with jump reflection governed by $F_0$,   driving function $x_0$, and the time delay at 0  described by \eqref{eq:delay}.
The critical values $0$ and $\infty$ of parameter $\rho$ mean zero delay and absorption at 0, respectively.

\subsection{Notation} Let $\mathbb R_+ := [0, \infty)$. Let $\mathcal C_T$ be the space of continuous functions $[0, T] \to \mathbb R$ and $\mathcal C$ be the space of continuous functions $\mathbb R_+ \to \mathbb R$ with topology of the uniform convergence on compact sets. For $T > 0$, let $\mathcal D_T$ be the Skorokhod space of
right-continuous functions with left limits  $[0, T] \to \mathbb R$, abbreviated as RCLL or c\'adl\'ag in French language, 
 and $\mathcal D$ be the Skorokhod space of RCLL functions $x : \mathbb R_+ \to \mathbb R$. We endow the spaces $\cD_T, \cD$ with Skorokhod's $J_1$-topology, see, for example, \cite{Whitt}. We will denote weak convergence of stochastic processes (in $\cC$ or $\cD$) by $X_n\Rightarrow X$. We define $a_+ := \max(a, 0)$ and $a_- := \max(-a, 0)$ for $a \in \mathbb R$. For a set $A$, we let $\overline{A}$ be its closure. Let $\mathrm{mes}(A)$ be the Lebesgue measure of $A$. For $T > 0$, let $\Lambda_T$ be a set of continuous one-to-one strictly increasing functions $\lambda : [0, T] \to [0, T]$. 

 For an RCLL function $h : \mathbb R_+ \to \mathbb R$, we define $\Delta h(t) ;= h(t) - h(t-)$. If $h$ has jump at time $t$, this is the size of this jump. If $h$ is continuous at time $t$, this is zero.

\subsection{Organization of this article} In Section 2, we complete the discussion of three reflection modes (Skorokhod's reflection, jump-type reflection, and delay). We state rigorous definitions of the switching process in Section 3. Next, we state the main result: Theorem~\ref{thm:main}, which includes both cases: sticky and generalized jump-type reflected processes.
In Section 4, we present applications of our results to various stochastic processes, including the ones in the previous articles.  Section 5 is devoted to the proof of Theorem~\ref{thm:main}. The Appendix contains proofs of a few technical lemmas. 

\subsection{Acknowledgment} The first author acknowledges 
  a partial support     by the National Research Foundation of Ukraine, project 2020.02/0014 
\textit{Asymptotic regimes of perturbed random walks: on the edge of modern and classical
probability}. Also, the first author was supported by the project \textit{Mathematical modelling of complex dynamical systems and processes caused by the state security} (Ukraine, Reg. No. 0123U100853). The second author thanks his Department for welcoming and positive atmosphere. 

\section{Background:   Skorokhod Reflection and  Boundary Controls}

In this  section we recall basic constructions and properties of reflected processes in order to get better understanding of a nature of  result that we obtain.

A {\it reflected Brownian motion} can be constructed by simply taking the absolute value $|W|$ of a Brownian motion $W$. However this construction can't be called  `natural'   for other  stochastic processes, and in particular, for non-symmetric L\'evy processes. In two 1961 articles \cite{Skorokhod1961a, Skorokhod1961b} \textsc{Anatoliy Skorokhod} developed a method to reflect any continuous function $B : [0, \infty) \to \mathbb R$, deterministic or stochastic, with $B(0) \ge 0$. He found a pair of 
continuous non-negative functions $X, L : [0, \infty) \to \mathbb R$ such that
\begin{equation}
\label{eq:basic-Skorohod}
X(t) = B(t) + L(t),\, t \geq 0,
\end{equation}
$L$ is non-decreasing and can increase only when $X = 0$; finally, $L(0) = 0$. In this problem~\eqref{eq:basic-Skorohod}, the function $B$ is called {\it driving} or {\it input}. See an example in Figure~\ref{fig:classic}.

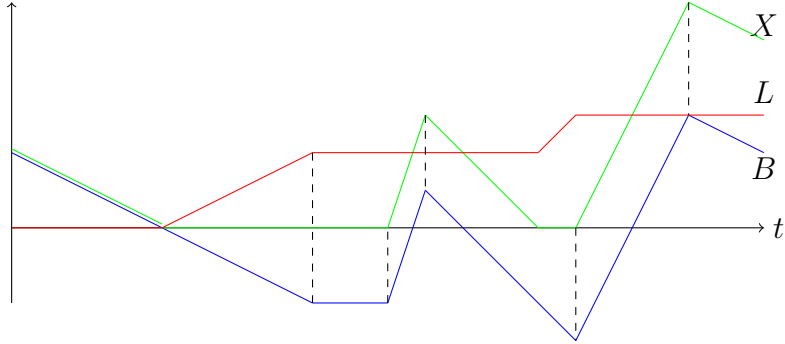
\begin{figure}[t]
\begin{tikzpicture}
\draw (10.2, 0) node {$t$};
\draw (10, 0.8) node {$B$};
\draw (10, 2.7) node {$X$};
\draw (10, 1.8) node {$L$};
    \draw[->] (0,0) -- (10,0); 
	\draw [->] (0,-1) -- (0,3);
	\draw [color = blue] (0, 1) -- (4, -1);
\draw [color = blue] (4, -1) -- (5, -1);
\draw [color = blue] (5, -1) -- (5.5, 0.5);
\draw [color = blue] (5.5, 0.5) -- (7.5, -1.5);
\draw [color = blue] (7.5, -1.5) -- (9, 1.5);
\draw [color = blue] (9, 1.5) -- (10, 1);
\draw [color = green] (0, 1.05) -- (2, 0.05);
\draw [color = green] (2, 0) -- (5, 0);
\draw [color = green] (5, 0) -- (5.5, 1.5);
\draw [color = green] (5.5, 1.5) -- (7, 0);
\draw [color = green] (7, 0) -- (7.5, 0);
\draw [color = green] (7.5, 0) -- (9, 3);
\draw [color = green] (9, 3) -- (10, 2.5);
\draw [color = red] (0, 0) -- (2, 0);
\draw [color = red] (2, 0) -- (4, 1);
\draw [color = red] (4, 1) -- (7, 1);
\draw [color = red] (7, 1) -- (7.5, 1.5);
\draw [color = red] (7.5, 1.5) -- (10, 1.5);
\draw [dashed] (4, 1) -- (4, -1);
\draw [dashed] (5, -1) -- (5, 0);
\draw [dashed] (5.5, 1.5) -- (5.5, 0.5);
\draw [dashed] (7.5, 0) -- (7.5, -1.5);
\draw [dashed] (9, 3) -- (9, 1.5);
 \end{tikzpicture}
     \caption{Classic Skorokhod reflection; $B$ is the original function; $X$ is the Skorokhod reflection; $L$ is the boundary term. 
     }
     \label{fig:classic}
    \end{figure}

The function $X$ has the same increments as $B$ while $X$ is positive, and the function $L$ pushes up only at instants when $X=0$ and have no effect otherwise. If $B$ is a Brownian motion, then $X$ has the same distribution as $|B|$. Thus, Skorokhod's definition is consistent with a naive approach to the notion of a reflected  Brownian motion. Moreover, it is well known that $L$ is the symmetric {\it local time} of $X$ at 0.

Later \textsc{Skorokhod}'s problem~\eqref{eq:basic-Skorohod} was generalized for functions $B\in \cD$; see, for example, \cite{DupuisIshii, Tanaka, mapping}.
The solution is given by the formula
\bel{eq:formula_eq_Sk}
L(t)=\sup\limits_{0 \leq s \leq t}(B(s)_-)=\Big(\inf\limits_{0 \leq s \leq t} B(s)\Big)_-,\ \ X(t)=B(t)+ \sup\limits_{0 \leq s \leq t}(B(s)_-).
\ee
The corresponding Skorokhod mapping $B\mapsto X$ is continuous in the Skorokhod space of RCLL functions (and is continuous in the subspace of continuous functions).  This allows us to prove functional limit theorems for heavy traffic limits as an application of continuous map theorem, see for example, \cite{Iglehart, Whitt}. For instant, the {\it Lindley's recursion} is nothing else but a solution to Skorokhod's  reflecting problem for random walks.

Assume again that $B$ is a Brownian motion, so that $X$ is a reflected Brownian motion. It is well known that $X$ spends zero time at 0 almost surely; that is, the Lebesgue measure 
of this time is zero:
\begin{equation}
\label{eq:zeros}
\mathrm{mes}(\{s\ge 0 \mid X(s) = 0\}) = 0\ \  \mbox{a.s.}
\end{equation}
Hence, the Skorokhod reflection  is {\it instantaneous.} There are other types of reflection and boundary behavior. A simple rule is alternatively called {\it absorbing}, or {\it stopping:} When the Brownian motion or any other input function (deterministic or stochastic) hits zero, it simply stops and stays constant after that. Another, more complicated rule, is {\it sticky reflection}. This version of a reflected Brownian motion is delayed when it hits zero.  Fix a parameter $\rho > 0$ called {\it delay rate} and consider the function $A(t) = t + \rho L(t)$ where $L$ is from~\eqref{eq:basic-Skorohod}. It is a continuous strictly increasing function and therefore it has an inverse $A^{-1}$. Plugging this inverse into the classic reflected function $X$, we get: 
\begin{equation}
    \label{eq:delay}
    \tilde{X}(t) = X(A^{-1}(t)),\quad A(t) = t + \rho L(t),\quad t \ge 0.
\end{equation} 

As a result, this process has same excursions as $X$ (but shifted in time) and spends positive time at zero, so~\eqref{eq:zeros} is no longer true.  The larger $\rho$ is, the longer the delay is. As $\rho \to \infty$, we get $A(t) \to \infty$, which corresponds to the  absorbed  process. When $\rho = 0$, we are back to the classic instantaneous reflection. 
We shall call this reflection {\it delayed}, as opposed to {\it instantaneous} classic Skorokhod reflection. However, this term should not be misunderstood.  This sticky reflected Brownian motion still leaves zero instantaneously  when it hits zero in the following sense. Notice  that the set of zeros of $\tilde X$ is a closed nowhere dense set a.s., so for any $t$ such that $\tilde X(t) = 0$ and  every $\varepsilon > 0$ there exists an $s \in (t, t + \varepsilon)$ with $\tilde X(s) > 0$.

The process $L$ is a  local time at 0 of the reflected Brownian motion $X$. Hence $A$ is a continuous additive functional of $X,$ and the general theory of Markov processes implies that the process $\tilde X$ is a strong Markov process. It can be proved that a sticky reflected Brownian motion is a  (weak) solution to the following stochastic differential equation:
\begin{align}
    \label{eq:RSDE-sticky}
    \begin{split}
\mathrm{d} X_\rho(t) &= \1_{\{X_\rho(t)>0\}}\, \mathrm{d} W(t) + \mathrm{d} L_\rho(t), \\
\mathrm{d} L_\rho(t) &= \rho^{-1}\1_{\{X_\rho(t)=0\}}\,\mathrm{d}t,            \end{split}
\end{align}
where $W$ is a standard Brownian motion. Note that increments of $X_\rho$ coincide with increments of $W$ when $X$ is positive. However, a reader should be careful:  As shown in \cite{Sticky-BM}, there is no strong solution  to the system of stochastic differential equations \eqref{eq:RSDE-sticky}. This is a very subtle and unexpected observation, because the process $\tilde X$ defined by \eqref{eq:delay} and \eqref{eq:formula_eq_Sk} is a function of $B.$

 Like the Skorokhod reflection map \eqref{eq:formula_eq_Sk}, formula  \eqref{eq:delay} can be used for RCLL functions or stochastic processes, including L\'evy processes. Moreover, if $X$ is a reflection of  spectrally positive L\'evy process or a reflected  diffusion, then the corresponding process $L$ will be a local time of $X$ (up to a multiplicative constant). So the process $\tilde X$ defined via time change \eqref{eq:delay} is naturally be called the delay at 0 of a Markov process $X$. See \cite{Amir, Bass} for further studies of a sticky reflected Brownian motion, and also also \cite{StickyLevy} on construction of sticky L\'evy processes (without reflection).

 The Skorokhod map $B\to X$ and the delayed Skorokhod reflection $B\mapsto \tilde X$ describe {\it continuous} exits from 0:  Discontinuity at the instant of exit from 0 can arise only from jumps in the driving process $B$, not the reflection itself.

There is  a discontinuous jump-type  reflection that corresponds to  non-local Feller-Wentzell boundary conditions in the semigroup theory of the diffusion processes  or to a jump entrance law in Ito's excursion theory. The first author in \cite{Pilipenko_Gen_refl}  proposed to consider (a deterministic)  {\it jump reflection} problem, when we replace $L$ in~\eqref{eq:basic-Skorohod} with $F(L)$. Here $F : [0, \infty) \to [0, \infty)$ is a given strictly increasing RCLL function with $F(0) = 0$ and $F(\infty) = \infty$:
\begin{equation}
\label{eq:general-Skorohod}
\bar X(t) = B(t) + F(L(t)),
\end{equation}
where $L$ is again a non-decreasing function that may increase only when $\bar X$ equals 0. More generally, in a later article \cite{Pilipenko_Gen_refl} it was shown that for every continuous $B$ there is a unique solution to this modified equation~\eqref{eq:general-Skorohod}, see formula \eqref{eq:gen-refl-fla}  in Lemma  \ref{lemma:existence} below for an explicit formula of a  solution. This explicit formula (without a formulation of a reflected problem) was used in \cite{Excursion} for a construction of Feller diffusions on a half-line, where $B$ was a Brownian motion and $F$ was a subordinator.

 Finally, note that this jump reflection can be combined with delay. Then we get sticky jump reflection. We do this in two steps, much like for sticky reflection $\tilde{X}$ above. First, we solve the equation~\eqref{eq:general-Skorohod} and get the jump reflection $\bar X$ without delay. Then we fix $\rho > 0$ and use $L$ and $\bar X$ to construct $\tilde{X}$ as in~\eqref{eq:delay}. The same trichotomy as above is present here: For $\rho = 0$, we are back in the case of jump reflection without delay. For $\rho \in (0, \infty)$, this is jump reflection with delay. Finally, as $\rho \to \infty$, we get absorbed  process. When it hits zero, it remains there forever.

\section{Definitions and the Main Result}

\subsection{Instantaneous jump reflection and delays} In this subsection, we provide explicit formulas for solutions on reflected problems, both classic~\eqref{eq:basic-Skorohod} and with jumps~\eqref{eq:general-Skorohod}. We state the definition again for completeness, although we discussed these reflection modes in the Introduction. 

\begin{defn}
Consider functions $x, F \in \mathcal D$ such that $x(0) \ge 0$, $F$ is strictly increasing, $F(0) = 0$, $F(\infty) = \infty$. A {\it solution to the generalized  Skorokhod problem} is a pair of functions  $y, l \in \mathcal D$, with the following properties: $y(t) \ge 0$ for all $t \ge 0$; 
$l $  is  non-decreasing function, and the following equality is true:
\begin{equation}
\label{eq:new-general}
y(t) = x(t) + F(l(t)),\quad t \ge 0.
\end{equation}
Moreover, the function $l$ can increase only when $y = 0$, i.e.
\begin{equation}
    \label{eq:restriction}
\int_0^\infty\1_{\{y(s)>0\}}\,\mathrm{d}l(s)=0.
\end{equation}
\label{defn:new-general}
We will call the function $x$ the {\it driving noise}, $y$ the {\it reflected process}, $F$ the {\it regulator},   $l$ the {\it boundary term}, and 
 denoted   $y = \mathcal S(x, F)$. 
\end{defn}


\begin{lemma}
\label{lemma:existence}
If $F$ is continuous or $x$ has no negative jumps, then there is a unique solution to \eqref{eq:new-general} and this solution is given by the formula
\begin{align}
\label{eq:gen-refl-fla}
\begin{split}
y(t) &= x(t) + F(l(t)) = x(t)+ F(F^{-1}(m(t)));\\
m(t) &= \sup\limits_{0 \leq s \leq t}(x(s)_-)=\inf\limits_{0 \leq s \leq t}(x(s))_-;\quad l(t) = F^{-1}(m(t)),
\end{split}
\end{align}
where $F^{-1}$ is the {\it generalized inverse} $F^{-1}(y) := \inf\{x\mid F(x) > y\}$.
\end{lemma}
For $F(t) \equiv t$, we get the classic Skorokhod reflection problem with the unique solution $y(t)  = x(t)+  m(t),$ see \cite{Tanaka, mapping}.
  If $F$ is continuous, we are back to classic Skorokhod reflection; in a way, the case of a continuous $F$ is not different from the case $F(t) = t$ for all $t \ge 0$. Indeed,  $F\circ F^{-1}(t)=t$ for strictly increasing and continuous  $F$. 
  Hence $y(t)  = x(t)+  m(t)$ again.
 \begin{remk}
The function  $m$  is continuous, since $x$ has no negative jumps. The function $F^{-1}$ is continuous and non-decreasing, since $F$ is strictly increasing, see \cite[Lemma 13.6.5.]{Whitt}.
Of course,  $F^{-1}$ is the standard inverse function if $F$ is strictly increasing and continuous.
\label{rmk:m-0}
\end{remk}

\begin{remk}
The condition that $x$ has no negative jumps or $F$ is continuous is important. Assume $x$ has, in fact, a negative jump at the point $t$. Then we can construct a setting when there is no solution.
Assume $y(t-) > 0$ for some $t$ but $\triangle x(t) < 0$ and $y(t-)+\triangle x(t)<0$. 
At time $t$, the jump of $x$ needs to be compensated by the function $F(l)$. If $F$ has jumps, then it may be impossible to find the compensation such that $y(t)=0$.
\end{remk}

 For continuous $x$, Lemma \ref{lemma:existence} 
 was shown in \cite{Pilipenko_Gen_refl}. The proof of the general case is postponed until the Appendix.  For convenience of readers we quote a result about the composition $F\circ F^{-1} $ used in \eqref{eq:gen-refl-fla},  taken also from  \cite{Pilipenko_Gen_refl}, 
and  illustrated in \textsc{Figure}~\ref{fig:nested}. 
\begin{lemma} Take a function $F\in \cD$ that is an increasing function with $F(0) = 0$ and $F(\infty) = \infty$. Consider the set $A_F := \mathbb R_+\setminus\overline{F(\mathbb R_+)}$. It is an open set, and therefore a countable union of intervals $A_F := \cup_{i}(\alpha_i, \beta_i)$. Then the function $G := F\circ F^{-1}$ satisfies:
$$
G(t) = 
\begin{cases}
    t,\, t \in \mathbb R_+\setminus A_F;\\
    \beta_i,\, t \in [\alpha_i, \beta_i).
\end{cases}
$$
\label{lemma:tech}
\end{lemma}

\begin{defn}
Fix a $\rho \in (0, \infty)$. Take the boundary term $l$ from Definition~\ref{defn:new-general} and define the time change $A(t) = t + \rho l(t)$ for $t \ge 0$. Plug $A^{-1}$ into the reflected function $y$ from Definition~\ref{defn:new-general}. Then $z(t) = y(A^{-1}(t))$ for $t \ge 0$ is called the {\it delayed jump-reflection} or {\it sticky jump-reflection}. The functions $x, F$ are called the {\it input} or {\it driving} function and the {\it regulator}, respectively, similarly to  Definition~\ref{defn:new-general}. The function $z$ called the {\it sticky jump-reflected function}, and the function $L$ defined as $L(t) = l(A^{-1}(t))$ for $t \ge 0$ is called the {\it sticky boundary term}. 
\end{defn}

We stress that a delayed Skorokhod reflection is not an alternative to jump reflection. Rather, these are two characteristics of a reflection: It can be with or without delay, and with or without jumps (that is, with $F$ continuous or discontinuous). All four options are possible. In addition, of course, we have absorbed  processes, but there $F$ does not matter anymore: It regulates only behavior at zero, and the resulting process stays at zero forever after it hits zero.

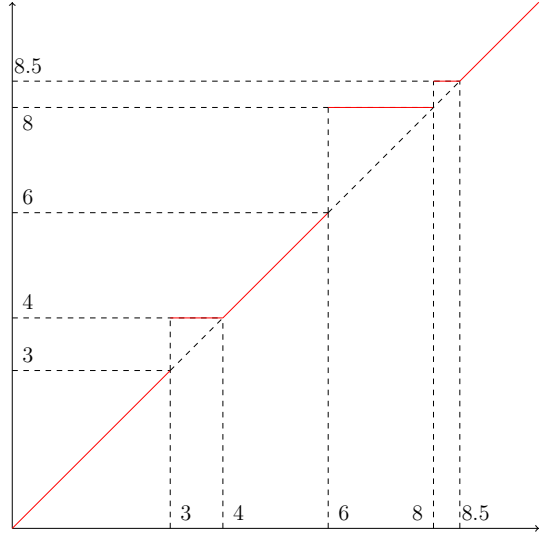
\begin{figure}[t]
\scalebox{0.7}{
\begin{tikzpicture}
\draw [->] (0, 0) -- (0, 10);
\draw [->] (0, 0) -- (10, 0);
\draw (0.3, 3.3) node {$3$};
\draw (3.3, 0.3) node {$3$};
\draw [dashed] (0, 3) -- (3, 3);
\draw [dashed] (3, 0) -- (3, 3);
\draw (0.3, 4.3) node {$4$};
\draw (4.3, 0.3) node {$4$};
\draw [dashed] (0, 4) -- (4, 4);
\draw [dashed] (4, 0) -- (4, 4);
\draw [color = red] (0, 0) -- (3, 3);
\draw [dashed] (3, 3) -- (3, 4);
\draw [dashed] (3, 3) -- (4, 4);
\draw [color = red] (3, 4) -- (4, 4);
\draw [color = red] (4, 4) -- (6, 6);
\draw (0.3, 6.3) node {$6$};
\draw (6.3, 0.3) node {$6$};
\draw [dashed] (0, 6) -- (6, 6);
\draw [dashed] (6, 0) -- (6, 6);
\draw (0.3, 7.7) node {$8$};
\draw (7.7, 0.3) node {$8$};
\draw [dashed] (0, 8) -- (8, 8);
\draw [dashed] (8, 0) -- (8, 8);
\draw (0.3, 8.8) node {$8.5$};
\draw (8.8, 0.3) node {$8.5$};
\draw [dashed] (0, 8.5) -- (8.5, 8.5);
\draw [dashed] (8.5, 0) -- (8.5, 8.5);
\draw [dashed] (6, 6) -- (6, 8);
\draw [color = red] (6, 8) -- (8, 8);
\draw [dashed] (8, 8) -- (8, 8.5);
\draw [dashed] (6, 6) -- (8.5, 8.5);
\draw [color = red] (8, 8.5) -- (8.5, 8.5);
\draw [color = red] (8.5, 8.5) -- (10, 10);
\end{tikzpicture}
}
\caption{Lemma~\ref{lemma:tech}. The function $G$ with $A_F = (3, 4)\cup(6, 8)\cup (8, 8.5)$.}
\label{fig:nested}
\end{figure}

\subsection{The switch problem} Here we define the regime switching problem in the general case, for arbitrary deterministic RCLL functions. This can be applied to Brownian motion, L\'evy processes, or any other stochastic processes.

\begin{defn}
Fix a {\it gap} $\delta > 0$, the {\it driving function} $x \in \mathcal D$, $x(0)\geq 0$, and the {\it regulating function} $F \in\mathcal D$  with $F(0) = 0$. A {\it solution to the switch problem} is a function $y \in \mathcal D$ with:
\begin{itemize}
\item times 
$0 = \tau_0 \leq \rho_1 \leq \tau_1 \leq \rho_2 \leq \ldots \leq \infty$
defined as $\rho_k := \inf\{t \geq \tau_{k-1}\mid y(t) \leq -\delta\}$ and $\tau_k := \inf\{t \geq \rho_k\mid y(t) \ge 0\}$ for $k = 1, 2, \ldots$;
\item two disjoint subsets $B = \cup_{k=1}^{\infty}[\rho_k, \tau_k)$, $A = \cup_{k=1}^{\infty}[\tau_{k-1}, \rho_k)$ of $\mathbb R_+$ such that $A\cup B = \mathbb R_+$ and the corresponding occupation times 
$$
T_A(t) = \mathrm{mes}(A\cap [0, t]),\quad T_B(t) = \mathrm{mes}(B\cap [0, t])
$$
(and therefore $T_A(t) + T_B(t) \equiv t$),
\end{itemize}
which satisfy
\bel{eq:main_eq_def_system}
y(t)=x(T_A(t))+F(T_B(t)), \ t\geq 0,
\ee
\label{defn:switch}
\end{defn}

\begin{remk}
We set $\inf\varnothing=+\infty$.
\end{remk}

\begin{remk}
From~\eqref{eq:main_eq_def_system} it follows that: $y(t) - F(T_B(t))$ is constant on each interval in $B$; and $y(t) - x(T_A(t))$ is constant on each interval in $A$.
\end{remk}

Let us explain this switch problem in plain English. 

\begin{enumerate}
\item We start with the input function $x$; the output function $y$ is equal to $x$ until it reaches below $-\delta$. This is the {\it normal regime} $A$. Note that the function $x$ can have jumps, so the input function $x$ (and together with it the output function $y$) might reach $(-\infty, -\delta]$ via a negative jump rather than a continuous path. Assume this happens at time $\rho_1$. This is the first piece of the input function.

\item Then we switch to the {\it boundary regime} $B$. We use the {\it regulating function} $F$, or simply {\it regulator}, starting from zero argument, to increase the output function $y$ until it reaches above $0$. The increments of the output function $y$ coincide with the increments of the regulator $F$. Again, note that the output function $y$ can reach $[0, \infty)$ by a jump rather than continuous movement. We stop at time $\tau_1$. So the regulator $F$ stops at $\tau_1 - \rho_1$. This is the first piece of the regulator. 

\item Next, we switch to the {\it normal regime} $A$ again. We govern the output function $y$ by the input function $x$: That is, the increments of the input and output functions coincide. We use the second piece of the input function $x$, starting from the point where we finished the first piece, in part 1. This happens until, as in part 1, the output function hits $(-\infty, -\delta]$. Assume this happens at the moment $\rho_2$.

\item Then we switch again to the {\it boundary regime} $B$ again, as in part 2. We govern the output function $y$ by the regulator $F$: The increments of the regulator and output functions coincide. We use the second piece of the regulator $F$, which starts at the time when the first piece ended. We continue until the output function hits $[0, \infty)$, and then switch to the normal regime again. 
\end{enumerate}

To summarize, we cut the graphs of input function $x$ and the regulator $F$ in pieces: We attach the first piece of $x$, then the first piece of $F$, then the second piece of $x$, then the second piece of $F$, and so on. 
By construction, for any pair $(x, F) \in \mathcal D^2$, and any $\delta > 0$, there exists a unique solution $y = \mathcal G_{\delta}(x, F)$.

See an example with a piecewise linear input function $x$ and another piecewise linear function as regulator $F$, and the output in \textsc{Figure}~\ref{fig:output}, where $\delta = 0.5$. 

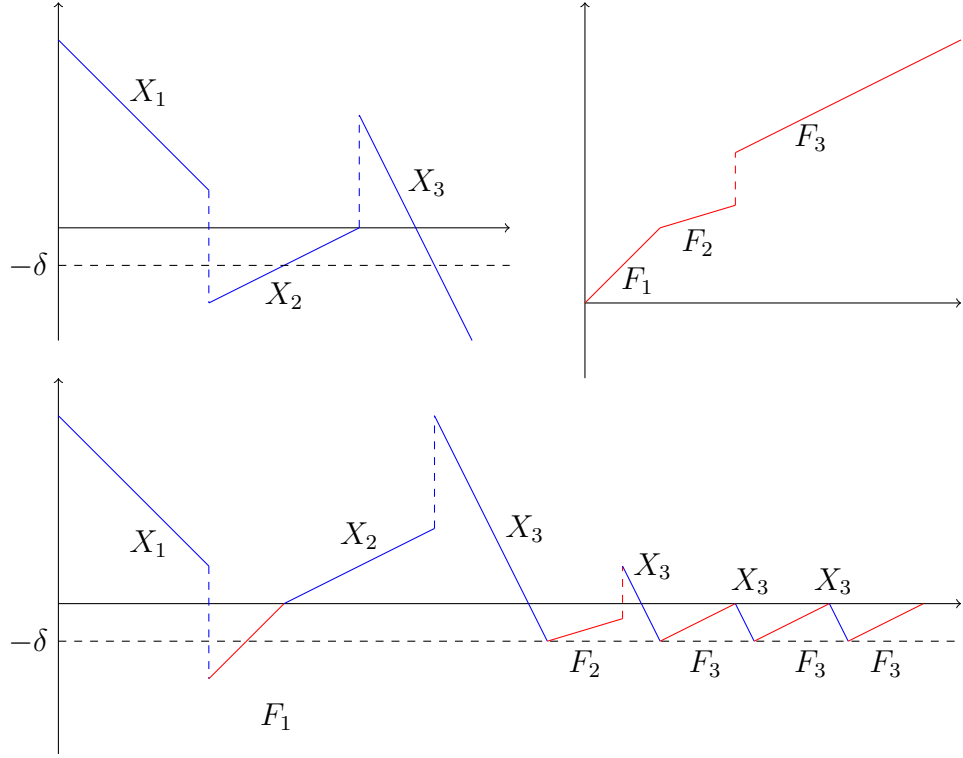
\begin{figure}[t]
\begin{tikzpicture}
\draw (1.2, 1.8) node {$X_1$};
\draw (3, -0.9) node {$X_2$};
\draw (4.9, 0.6) node {$X_3$};
\draw (-0.4, -0.5) node {$-\delta$};
\draw [->] (0, -1.5) -- (0, 3);
\draw [->] (0, 0) -- (6, 0);
\draw [color = blue] (0, 2.5) -- (2, 0.5);
\draw [dashed, color = blue] (2, 0.5) -- (2, -1);
\draw [color = blue] (2, -1) -- (4, 0);
\draw [dashed, color = blue] (4, 0) -- (4, 1.5);
\draw [color = blue] (4, 1.5) -- (5.5, -1.5);
\draw [dashed] (0, -0.5) -- (6, -0.5);
\draw [->] (7, -2) -- (7, 3);
\draw [->] (7, -1) -- (12, -1);
\draw [color = red] (7, -1) -- (8, 0);
\draw [color = red] (8, 0) -- (9, 0.3);
\draw [dashed, color = red] (9, 0.3) -- (9, 1);
\draw [color = red] (9, 1) -- (12, 2.5);
\draw (7.7, -0.7) node {$F_1$};
\draw (8.5, -0.2) node {$F_2$};
\draw (10, 1.2) node {$F_3$};
\draw [->] (0, -7) -- (0, -2);
\draw [->] (0, -5) -- (12, -5);
\draw [dashed] (0, -5.5) -- (12, -5.5);
\draw [color = blue] (0, -2.5) -- (2, -4.5);
\draw [dashed, color = blue] (2, -4.5) -- (2, -6);
\draw [color = red] (2, -6) -- (3, -5);
\draw [dashed, color = blue] (5, -4) -- (5, -2.5);
\draw [color = blue] (3, -5) -- (5, -4);
\draw [color = blue] (5, -2.5) -- (6.5, -5.5);
\draw [color = red]  (6.5, -5.5) -- (7.5, -5.2);
\draw [dashed, color = red] (7.5, -5.2) -- (7.5, -4.5);
\draw [color = blue] (7.5, -4.5) -- (8, -5.5);
\draw [color = red] (8, -5.5) -- (9, -5);
\draw [color = blue] (9, -5) -- (9.25, -5.5);
\draw [color = red] (9.25, -5.5) -- (10.25, -5);
\draw [color = blue] (10.25, -5) -- (10.5, -5.5);
\draw [color = red] (10.5, -5.5) -- (11.5, -5);
\draw (1.2, -4.2) node {$X_1$};
\draw (2.9, -6.5) node {$F_1$};
\draw (4, -4.1) node {$X_2$};
\draw (6.2, -4) node {$X_3$};
\draw (7, -5.8) node {$F_2$};
\draw (7.9, -4.5) node {$X_3$};
\draw (8.6, -5.8) node {$F_3$};
\draw (9.2, -4.7) node {$X_3$};
\draw (10, -5.8) node {$F_3$};
\draw (10.3, -4.7) node {$X_3$};
\draw (11, -5.8) node {$F_3$};
\draw (-0.4, -5.5) node {$-\delta$};
\end{tikzpicture}
\caption{The switch problem. Top left: Input function $x$. Top right: regulator $F$. Bottom: Output $y = \mathcal S_{\delta}(x, F)$. All functions are right-continuous.}
\label{fig:output}
\end{figure}

Generally, we cannot consider the switch problem for $\delta = 0$ and switch the regime when the process enters and exits $(0,\infty)$ (or enters and exits $[0,\infty)$). The only difficulty is treating  the switch problem  at 0. For example, it is unclear how to treat the definition if regime $A$   pushes  down and regime $B$ pushes up, or if functions $x$ and $F$ behaves like excursions of a Brownian motion having no intervals of monotonicity. We will define the solution to the switching problem $\mathcal G_{\delta}(x, F)$ with $\delta=0$ if contradictions in the definition do not appear. 
The following cases are examples:
\begin{enumerate}[label = (\alph*)]
\item $x$ and $F$ are step-functions with finitely many jumps in any $[0, t]$; regime $A$ is selected  if $y(t)> 0$; regime $B$ is selected if $y(t)\leq  0$.
\item $F$ is a non-decreasing step function with finitely many jumps in any $[0, t]$; regime $A$ is selected if $y(t)>0$; regime $B$ is selected if $y(t)\leq 0$.
\item Same as above, where  regimes $A$ or $B$ are selected  if $y(t)\geq  0$  or $y(t)< 0,$ respectively.  
\item Any   function $y$ constructed
from pieces of  $x$ and $F$ under condition that $y$
never hits 0 and has  finitely many crossings of 0  during any $[0, t]$.
\end{enumerate}
We will apply the switch  problem $\mathcal G_{0}(x, F)$ to the cases when:
\begin{enumerate}[label = (\alph*)]
    \item $x$ and $F$ are independent compound Poisson processes;
    \item $ x(t)=x_{[t]}, F(t)=F_{[t]}, $ where $(x_n)$ and $(F_n)$ are independent random walks;
    \item $x$ is a L\'evy process, and $F$ is a jump-type subordinator with finite  L\'evy measure.
\end{enumerate}
 
\begin{remk}
 If we consider the problem  $y=\mathcal G_{0}(x, F)$, then we will always assume that $x$ and $F$ are such that the switch problem $\mathcal G_{0}(x, F)$  is well defined.
\end{remk}

\subsection{Main results} 
 
For a non-decreasing function $f : \mathbb R_+ \to \mathbb R$ we say that $s \in \mathbb R_+$ is a {\it growth point} if $f(t) > f(s)$ for $t > s$ and $f(t) < f(s)$ for $t < s$.
\begin{thm}\label{thm:main}
Let $(x_n), (F_n)$ be two sequences in $\mathcal D$. Define 
\begin{equation}
    \label{eq:running-maximum}
    m_0(t):=\sup\limits_{s\in[0,t]}(x_0(s)_-).
\end{equation}
Take two sequences $(\delta_n)$ or $(\varrho_n)$ of real numbers such that for every $n$ we have $\delta_n \ge 0$ and $\varrho_n > 0$, $\delta_n \to 0$ and $\varrho_n \to \varrho \in [0, \infty]$. Consider a sequence $(y_n) \subseteq \mathcal D$ of functions defined as
$$
y_n  := \mathcal G_{\delta_n}(x_n, F_n(\cdot/\varrho_n)),\, t \ge 0.
$$
Assume that

\begin{enumerate}[label = (\alph*)]
    \item\label{ass.a} $x_n(0)\geq 0, F_n(0)=0,$ $n = 0, 1, 2, \ldots$;
    \item \label{ass.b}  $x_n \to x_0$ in $\mathcal D$ and $F_n  \to F_0 $ in $\mathcal D$; 
    \item \label{ass.c}  $F_0$ is strictly increasing, $F_0(\infty) = \infty$;
    \item\label{ass.d}  $x_0$ does not have negative jumps;
    \item \label{ass.e} if $\alpha \ge 0$ and $t\geq 0$ are such that $F_0(\alpha) \ne F_0(\alpha-) = m_0(t)$, where $m_0$ is defined in~\eqref{eq:running-maximum}, then $t$ is a growth point of $m_0$. 
\end{enumerate}

\noindent Then $y_n \to y_0$ in $\mathcal D$, where $y_0$ depends on $\varrho$:

\begin{enumerate}
\item {\normalfont Classic reflection:} $\varrho = 0$. Then $y_0 =\mathcal S(x_0, F_0)$. 
\item {\normalfont Delayed reflection:} $\varrho\in (0,\infty)$. Then 
$$
y_0(t)=\mathcal S(x_0, F_0)(A_\varrho^{-1}(t)),\quad A_\varrho(t):= t +\varrho F_0^{-1}(m_0( t) ).
$$
\item {\normalfont Absorption:} $\varrho=\infty$. Then  $y_0(t)=\mathcal S(x_0, F_0)(t\wedge \sigma)$, assuming
$$
\sigma := \inf\{s\geq 0\ | \ x_0(s)=0\}=\inf\{s\geq 0\ | \ x_0(s)<0\}.
$$
\end{enumerate}
\end{thm}
 \begin{remk}
     For the case $\varrho = \infty$, only conditions (a), (b), and (d) in this theorem suffice. 
 \end{remk}
 \begin{remk}
 Notice that $\lim_{\varrho\to0+}A^{-1}_\varrho(t)=t$ and 
 $\lim_{\varrho\to+\infty}A^{-1}_\varrho(t)=t\wedge \sigma$, so the limit function $y_0$ can be formally written as $\mathcal S(x_0, F_0)(A_\varrho^{-1}(t))$  in 
 cases $\varrho=0$ and $\varrho=\infty$ too.
  \end{remk}
\begin{remk}
The regime switch reminds a  {\it penalty method}: a  reflected Brownian motion on the half-line can be obtained as  a limit of  Brownian motions on $\mbR$ with a large positive drift in a neighborhood of 0 or below 0, for example, $n^21_{[0, 1/n)}(x)$ or $n 1_{(-\infty, 0]}(x)$, which pushes to $[0, \infty)$ as $n$ increases, see, for example, \cite{Bruggeman, Slo}. 
\end{remk}

\begin{corl}\label{corl:main}
Let $(x_n), (F_n)$ be two sequences of RCLL processes that for almost every $\omega$ satisfy assumptions \ref{ass.a}, \ref{ass.c}, \ref{ass.d}, and \ref{ass.e} of Theorem \ref{thm:main}. Assume also that  sequences  $(\delta_n)$, $(\varrho_n)$, and $(y_n)$ are defined as in Theorem \ref{thm:main}, and also 
$(x_n, F_n)\Rightarrow  (x_0, F_0)$ as $n\to\infty$ in $\mathcal D$. Then we have convergence in distribution
\bel{eq:307}
y_n := \mathcal G_{\delta_n}(x_n, F_n(\cdot/\varrho_n))\Rightarrow S(x_0,F_0)(A_\varrho^{-1}).
\ee
where $y_0$ depends on $\varrho$ similarly to Theorem \ref{thm:main}. 
\end{corl}
\begin{proof}
By the Skorokhod representation theorem, there are copies $(\wt x_n, \wt F_n)\overset{d}= (x_n, F_n), n\geq 0$
such that we have convergence almost surely:
\[
(\wt x_n, \wt F_n)\to (\wt x_0, \wt F_0),\quad n\to\infty.
\]
Hence almost surely we have convergence $\mathcal G_{\delta_n}(\wt x_n, \wt F_n(\cdot/\varrho_n))\rightarrow \cS(\wt x_0,\wt F_0)(\wt A_\varrho^{-1}), n\to\infty$. This latter convergence implies \eqref{eq:307}. 
\end{proof}

\section{Examples and Applications}

\subsection{Driving Brownian motion} 
Assume that $x_n=w, n\geq 1$, where $w$ is a Brownian motion, $\varrho_n=1, n\geq 1,$   $F_n(t)=at,$ where $a>0,$ and $(\delta_n)\subset(0,\infty)$ be any sequence of positive numbers that converges to 0. That is, the process $y_n$ is a continuous process that moves like a Brownian motion before hitting $-\delta_n$, then switches a regime and moves up with constant speed until hitting 0, then switches regime for a Brownian motion until hitting $-\delta_n$, etc.   We have
$x_0(t)=w(t),\ F_0(t)=at, \ F_0(F_0^{-1}(t))= t$. Then $\mathcal S(x_0, F_0)=\mathcal S(x_0)=\mathcal S(w)=w+m_0 $ is the reflected Brownian motion, and $A_\varrho(t) := t+\varrho F_0^{-1}(m_0(t)) = t + a^{-1}m_0(t)$. Hence $y_0=(w+m_0)(A_\varrho^{-1})$ is a sticky reflected Brownian motion. If we assume $F_n(t)=a_nt,$ where $\lim_{n\to\infty}a_n=+\infty,$ i.e., the regime below 0 pushes up strongly, then the limit process will be the usual reflected Brownian motion without any delay.

Note that the process $y_n$ has the same distribution as a solution to the stochastic equation
\[
\mathrm{d} Y_n(t) =\1_{\{t\in A_n\}}\,\mathrm{d}w(t) + a\,\1_{\{t\in B_n\}}\,\mathrm{d} t, \ t\geq 0,
\]
with $Y_n(0)=0.$ Here sets $A_n$ and $B_n$ are defined as follows:
\begin{align*}
A_n &= \cup_{k=0}^\infty [\tau^{(n)}_k, \rho^{(n)}_{k+1}), \quad 
B_n = \cup_{k=1}^\infty [\rho^{(n)}_k, \tau^{(n)}_k), \\
\tau^{(n)}_0 &:= 0, \quad \rho^{(n)}_{k+1}:=\inf\{t\geq \tau^{(n)}_k\ :\ Y_n(t)\leq -\delta_n\},\quad \tau^{(n)}_{k}:=\inf\{t\geq \rho^{(n)}_k\ :\ Y_n(t)\geq 0\}.
\end{align*}

\subsection{Storage processes}
Take   Poisson processes  ${N_{ \lambda_n}(t)}$ and $\hat N_{\hat \lambda_n} (t)$  with intensities $\lambda_n$ 
and $\hat \lambda_n$,  respectively. Take sequences of non-negative i.i.d. random variables $(\xi_k^{(n)})_{k\geq 1}$ and   $(\zeta_k^{(n)})_{k\geq 1}$. We will assume that all processes and sequences are jointly independent.
Consider two compound Poisson processes $x_n$ and $F_n$ with positive jumps, where $x_n$ has an additional negative drift term:
$$
x_n(t)=\sum\limits_{k=1}^{N_{ \lambda_n} (t)}\xi_k^{(n)}-r_nt,\quad F_n(t)=\sum_{k=1}^{\hat N_{\hat \lambda_n} (t)}\zeta_k^{(n)}
$$
Note that $F_n$ is a jump-type subordinator. Then the process  $y_n(t)=\mathcal G_{0}(x_n,  F_n)$ is a Markov storage process whose behavior can be described as follows: 

\begin{enumerate}
\item If $y_n(t)>0$, then it  has a negative drift at rate $r_n$  and jumps up in a Poisson clock with intensity $\lambda_n$; the value of the $k$th jump is  $\xi_k^{(n)}$.
\item If $y_n(t)=0$, then $x_n$ stays at 0 the exponential time $\mathrm{Exp}(\hat \lambda_n)$ and then jumps up with the distribution  $\zeta_k^{(n)}$.
\end{enumerate}

It is well  known that under assumptions
\begin{align*}
\lim_{n\to\infty}(\lambda_n\E  \xi_k^{(n)} -r_n) &= \mu,\\
\lim_{n\to\infty} \lambda_n\E\left[(\xi_k^{(n)})^2\right] &= \sigma^2>0,\\
\lim_{n\to\infty}\lambda_n\E\left[(\xi_k^{(n)})^2\1_{\{\xi_k^{(n)}>\ve\}}\right] &= 0,\quad \ve > 0;
\end{align*}
we have convergence in distribution $x_n(t)\Rightarrow \mu t+\sigma w(t)$ as $n\to\infty$ in $\cD,$ where $w$ is a standard Brownian motion. Assume that there are  normalizing constants  $\gamma_n$ such that
\bel{eq:352}
 \sum_{k=1}^{[\gamma_n t]}\zeta_k^{(n)}\Rightarrow t, \quad n\to\infty.
\ee
Suppose that there exists a limit $\hat\lambda_n/\gamma_n \to \varrho \in [0,\infty]$. Without loss of generality we may assume that $\hat N_{\hat \lambda_n}(t)= \hat N({\hat \lambda_n} t),$ where $\hat N$ is a fixed Poisson process with intensity 1. It is well known that $\hat N({\hat \gamma_n}t)/\hat\gamma_n \to t$ uniformly on every $[0, T]$ a.s. Thus
\[
\sum_{k=1}^{\hat N(\hat \lambda_n t)}\zeta_k^{(n)}=
\sum_{k=1}^{\hat N(\gamma_n t/\varrho_n)}\zeta_k^{(n)},
\]
where $\varrho_n= \gamma_n/{\hat{\lambda}_n}$.

The application of Corollary~\ref{corl:main} with $\delta_n:=0$ implies the following result.

\begin{enumerate}
\item If $\varrho = 0$, then the limit of $y_n$ is the Skorokhod reflection of $\mu t+\sigma w(t)$ at 0.
\item If $\varrho > 0$, then the limit of $y_n$ is  the sticky reflection of $\mu t+\sigma w(t)$ at 0.
\item If $\varrho=\infty,$ then the limit of $y_n$ is $\mu t+\sigma w(t)$  stopped at 0.
\end{enumerate}

In addition, the case when $ \xi_k^{(n)}\overset{d}=\zeta_k^{(n)}$ was considered in  \cite{HarrisonLemoine}. In this case \eqref{eq:352} is satisfied with $\gamma_n = \lambda_n/r_n$. Another example when  \eqref{eq:352} holds is the following one: $\zeta_k^{(n)}=\zeta_k/\gamma_n$ where $(\zeta_k)$ are non-negative independent identically distributed random variables, $\E  \zeta_k=1$, and $(\gamma_n)$ is any sequence of positive numbers such that $\gamma_n \to +\infty$.

 If the sequence $(\sum_{k=1}^{[\gamma_n t]}\zeta_k^{(n)})$ converges to an increasing subordinator $F_0(t), t\geq 0$, then the limit process will be a diffusion with jump-type exit from 0. The only non-triviality is to verify condition (e) of Theorem~\ref{thm:main}. This will be done in more general case in next example.

\subsection{Convergence to a reflected L\'evy  process with a delay at 0} Let a sequence of positive numbers $(\rho_n)$ and  sequences of stochastic 
processes $(x_n), (F_n)$ be such that
\begin{enumerate}
\item $ (x_n, F_n) \Rightarrow  (x_0, F_0)$ as $n\to\infty;$
\item $x_0$ is a L\'evy process without negative jumps;
\item $F_0$ is an increasing subordinator;
\item $x_0$ and $F_0$ are independent;
\item $\varrho_n \to \varrho \in [0,\infty].$ 
\end{enumerate}
The process $T_0 := m_0^{-1}$ is a subordinator (may be killed at an exponential time), see \cite[Theorem 1, p.189]{Bertoin}. 
A point $t$ is not  a point of growth of $m_0$ if and only 
 if  $m_0(t)$ is a point of jump of  $T_0$.
 At any fixed (non-random) $\alpha \geq 0$, the function $T_0$ is almost surely continuous. Since the processes $F_0$ and $T_0$ are independent and the set of jumps of $F_0$ is at most countable, condition (e) of Theorem~\ref{thm:main} is true almost surely. The application of Corollary \ref{corl:main} implies convergence in distribution \eqref{eq:307}. It is natural to say that the  process $S(x_0,F_0)(A_\varrho^{-1})$ is  the process $x_0$ that having jump-type reflection at 0 with a delay. It may be seen similarly to \cite[Chapter II.3 (c)]{Excursion}  that the  process  $\cS(x_0,F_0)$ is a Markov process 
and $L(t)=F_0^{-1}(m_0(t)), t\geq 0$ is its local time at 0
if  $x_0$ and $F_0 $ are independent.

\subsection{Perturbed random walks} 
Let $(\xi_k)$ be a sequence of independent identically distributed mean-zero random variables with finite variance $\sigma^2>0.$ Consider the random walk $S_\xi(n):=\sum_{k=1}^n\xi_k$, where $S_\xi(0):=0$. Let us extend  $S_\xi$ to non-negative half-line as follows  $S_\xi(t):=S_\xi([t])$ for $t\ge0.$ It is well known that 
\[
\frac{S_\xi(nt)}{\sqrt{n}}\Rightarrow \sigma w(t),\quad  n\to\infty
\]
in $\cD$ by the Donsker theorem, where $w$ is a standard Wiener process. Assume that a non-negative random variable $\eta$ belongs to the domain of attraction of the $\beta$-stable law with $\beta\in(0,1)$. Consider  a Markov chain  $(X(n))$ with transition probabilities
\[
\Pb(X(n+1)=y | X(n)=x)=\begin{cases}
\Pb(\xi=y-x),\ x\geq 0;\\
\Pb(\eta=y-x),\ x< 0.
\end{cases}
\]
We will interpret $X$ as a perturbation of $S_\xi$  below 0. Note that $X$ has the same distribution as a solution of two-phase system, where $\delta=0$, $x = S_\xi$, and $F = S_\eta$. Here $S_\eta(n)=\sum_{k=1}^n\eta_k,  $  $S_\eta(t):=S_\eta([t]) $ for $ t\geq 0,$ and the sequences  $(\eta_k)$ and $(\xi_k)$ are independent. That is,
\[
X\overset{d}= \cG_0(S_\xi, S_\eta).
\]
Note that natural scaling for $S_\eta$ is not $\sqrt{n}$ as in the Donsker theorem. There is a sequence $(a(n))$ that is slowly varying at infinity with index $\frac{1}{\beta}$ such that
\[
\frac{S_\eta( n\cdot)}{a(n)}\Rightarrow   U_\beta(\cdot),\quad  n\to\infty,
\]
in $\cD$, where $U_\beta$ is a $\beta$-stable subordinator: a non-decreasing L\'evy process with Laplace transform $\E[\exp(-\lambda  U_\beta(t))] = \exp(-t\lambda^{\beta})$ for $t, \lambda \ge 0$. It can be seen that there is a sequence $(b(n))$ that is slowly varying at infinity with index $\beta/2$ such that
\[
\frac{S_\eta( b(n) \cdot )}{\sqrt{n}} 
\Rightarrow U_\beta(\cdot),\quad n\to\infty.
\]
Define $\varrho_n := b(n)/n$. Notice that $\lim_{n\to\infty}\varrho_n = 0$ because $\beta/2<1$. Thus
\[
\frac{S_\eta(\varrho_n n\cdot)}{\sqrt{n}} =\frac{S_\eta( b(n) \cdot )}{\sqrt{n}} \Rightarrow U_\beta(\cdot),\quad n\to\infty.
\]
Define processes $x_n(t):= S_\xi(nt)/\sqrt{n}$ and $F_n(t):= S_\eta(\varrho_n nt)/\sqrt{n}$. Hence
\[
 \Big(\frac{X( nt )}{\sqrt{n}}\Big)_{t\geq0}  \overset{d}= \Big(\cG_0\left(x_n,F_n(\cdot/\varrho_n)\right)(t)\Big)_{t\geq0}.
\]

Corollary~\ref{corl:main} implies 
the weak convergence
\bel{eq:IPP_conv_refl_stable}
\frac{X([n\cdot ])}{\sqrt{n}}
\overset{d}= \cG_0\left(x_n, F_n(\cdot/\varrho_n)\right)  \Rightarrow  
\cS\left(  w,   U_\beta \right)   = w +  U_\beta\circ U_\beta^{-1}\circ m,   \ee
where $w$  and $U_\beta$ are independent, $m(t) = \max_{s\in [0,t]}(w(t)_-)$. Feasibility of condition (e) of Theorem~\ref{thm:main} follows from the reasoning of the previous example.

\begin{remk}
If we multiply each $b(n)$ by a constant $C > 0$, then 
$$
S_\eta\left(\frac{ C b(n) t  }{\sqrt{n}}\right) \Rightarrow \hat{U}_{\beta}(t) := U_\beta(Ct).
$$
However, the result  will be unchanged because 
$\hat{U}_\beta\circ\hat{U}_\beta^{-1} = U_\beta\circ U_\beta^{-1}$.
\end{remk}

Note that the same result is also true if transition probabilities for $X$ are
\[
\Pb(X(n+1)=y | X(n)=x)=\begin{cases}
\Pb(\xi=y-x),\ x> 0;\\
\Pb(\eta=y-x),\ x\leq 0.
\end{cases}
\]
Convergence \eqref{eq:IPP_conv_refl_stable} was proved in \cite{Iksanov} and for a particular case in \cite{PilipenkoPrykhodko}.   The same convergence was obtained in \cite{Iksanov} for the Donsker scaling limits of sequences
\[
   X'(n+1)=
\begin{cases}
  X'(n) + \xi_{n+1}, &   X'(n) > 0, \\
\eta_{n+1}, & \hat X'(n) \leq 0;
\end{cases}
\]
 \[
\grave X(n+1)=
\begin{cases}
\grave X(n) + \xi_{n+1}, & \grave X(n)>0 \mbox{ and } \grave X(n) + \xi_{n+1} > 0, \\
0, & \grave X(n)>0 \mbox{ and } \grave X(n) + \xi_{n+1} \leq 0,\\
\eta_{n+1}, & \grave S_v(n)=0,
\end{cases}
\]
that also satisfy assumptions of Corollary \ref{corl:main}  (the corresponding additional reasoning can be found in \cite{Iksanov}). 
Some ideas used in this paper are taken from \cite{Iksanov, PilipenkoPrykhodko}, but only now it becomes clear how different scaling for $(\xi_k)$ and $(\eta_k)$ interplay and give  a limit for the Donsker scaling of perturbed random walk $(X(n))$.
\begin{remk}
 Consider  sequences $(X_l(k))_{k\geq 0}$ that have the same transition probabilities as $(X(k))_{k\geq 0}$ but different initial values such that $X_l(0)/\sqrt{l} \Rightarrow \zeta, n\to\infty$. It is easy to see  that  the following convergence  can be proved:
 \[
X_l(t)/\sqrt{l} \Rightarrow \zeta+w(t)+U_\beta \circ U_\beta^{-1} \circ ((-\zeta+m(t))_+),\quad t\geq 0,
 \]
 where processes $w$, $U_\beta$, and the random variable $\zeta $ are independent.
\end{remk}
 
\section{Proof of the Main Result}

This section is organized as follows. In Subsection 5.1, we state an upper and a lower bound for the solution of the switch problem. These two lemmas have proofs postponed until the Appendix. In Subsection 5.2, we write four technical convergence lemmas used for the proof of Theorem~\ref{thm:main}. The first two of these four lemmas, similarly, have proofs in the Appendix. The other two lemmas are quoted from other sources so we do not give their proofs. In the next three subsections,  we prove Theorem~\ref{thm:main} for three cases: $0 < \varrho < \infty$, $\varrho = 0$, and $\varrho = \infty$.

\subsection{Estimates for the switch problem} Take a constant $\delta > 0$ and functions $x, F$, as in the definition of the switch problem. Let $y = \mathcal G_{\delta}(x, F)$ be the solution, and $A, B$ the corresponding sets. Define the running maximum
$$
m(t) = \sup\limits_{s \in [0, t]}(x(s)_-),\, t \ge 0.
$$
We state the two lemmas which together form the basis for the proof. 

\begin{lemma} Let $y=\cG_\delta(x,F)$. Then for  every $t \geq 0$, we have: $F(T_{B}(t)-) \leq m(T_{A}(t))$. 
\label{lemma:est-1}
\end{lemma}

\begin{lemma} Let $y=\cG_\delta(x,F)$. Then for  every  $t \geq 0$, we have:
\begin{align*}
F(T_{B}(t)) \geq  -\delta + m(T_{A}(t)) & -  \sup_{s \in[0,T_{A}(t)]}(x(s-)-x(s)) \\ & - \sup_{0\leq s_1\leq s_2\leq T_{B}(t)}(F(s_1)-F(s_2)).
\end{align*}
\label{lemma:est-2}
\end{lemma}
 
\subsection{Preliminary results} 
The next two lemmas are technical convergence results, with (simple) proofs postponed until the Appendix.

\begin{lemma} If $x_n\to x_0$ in $\mathcal D$, where $x_0$ has only positive jumps, then for any $T > 0$, $\sup_{s \in[0,T]}(x_n(s-) - x_n(s)) \to 0$. 
\label{lemma:jumps}
\end{lemma}

\begin{lemma} If $F_n\to F_0$ in $\mathcal D$, where $F_0$ is a non-decreasing function, then for any $T > 0$
$\sup_{0\leq s_1\leq s_2\leq T }(F_n(s_1) - F_n(s_2)) \to 0$. 
\label{lemma:increase}
\end{lemma}

Recall  a classical characterization of convergence in a Skorokhod space from  \cite[Chapter 3, Proposition 6.5, page 125]{Ethier-Kurtz}:

\begin{lemma}
\label{prop:criterion}  
We have $x_n \to x_0$ in $\mathcal D$ if and only if for any non-negative $t_n \to t_0$, we have:
\begin{enumerate} 
\item \label{prop:criterion:1} all limit points of $(x_n(t_n))_{n\geq1}$ are either $x_0(t_0)$ or $x_0(t_0-)$;
\item \label{prop:criterion:2} if $x_n(t_n) \to x_0(t_0)$, $s_n \geq t_n$, $s_n \to t_0$, then
$x_n(s_n) \to x_0(t_0)$;
\item \label{prop:criterion:3} if $x_n(t_n) \to x_0(t_0-)$, $s_n \leq t_n$, $s_n \to t_0$, then $x_n(s_n) \to x_0(t_0-)$.
\end{enumerate}
\end{lemma}

The following lemma easily follows from the previous one.

\begin{lemma} Assume $x_n \to x_0$ in $\mathcal D$ and $y_n \to y_0$ in $\mathcal D$. If for every point $t \ge 0$ at least one of two functions $x_0$ and $y_0$ is continuous at this point, then  $x_n + y_n \to x_0 + y_0$ in $\mathcal D$.
\label{lemma:sum}
\end{lemma}

\subsection{Proof of   Theorem \ref{thm:main} for $\varrho\in(0,\infty)$.} {\it Step 1.} For each $n = 1, 2, \ldots$, denote 
\begin{equation}
\label{eq:m-n-sup}
m_n(t) = \sup\limits_{s \in [0, t]}(x_n(s)_-),\, t \ge 0.
\end{equation}
Define $A_n$ and $B_n$ to be the sets $A$ and $B$ for the $n$th switch problem. 

First, the sequences $(T_{A_n})$, $(T_{B_n})$ are pre-compact in $\mathcal C_T$ for every $T > 0$. Indeed, $(T_{A_n})$, $(T_{B_n})$ are globally Lipschitz continuous with Lipschitz constant 1; therefore, this sequence $(T_{A_n})$ is equicontinuous in $\mathcal C_T$ for every $T > 0$. By the Arzela-Ascoli theorem, these sequences are pre-compact. The same applies to $(T_{B_n})$. Take a subsequence $(n')$ such that 
\begin{equation}
\label{eq:uniform-conv}
T_{A_{n'}} \to S_0,\quad T_{B_{n'}} \to S_1,\quad \text{ in any } \quad \mathcal C_T,
\end{equation} 
where $S_0(t) + S_1(t) = t$. Since $x_n \to x_0$ in $\mathcal D_T$ we have convergence $m_n\to m_0$ in $\mathcal D_T$. Recall that $m_0$ is continuous, see Remark~\ref{rmk:m-0}. Thus we have the uniform convergence $m_n\to m_0$ on any $[0, T]$. Combining these facts of the uniform convergence and the observation
\begin{equation}
\label{eq:values-0-t}
0 \le T_{A_{n'}}(t) \le T,\quad 0 \le T_{B_{n'}}(t) \le T\quad \mbox{for}\quad t \in [0, T].
\end{equation}
we get the uniform convergence on any $[0,T]$:
\begin{equation}
    \label{eq:composition}
m_{n'}(T_{A_n'} ) \to m_0(S_0 ), 
\end{equation}

{\it Step 2.} From Lemma \ref{prop:criterion}, for any $t>0$ all limit points of $F_{n'}(\varrho_n^{-1}T_{B_{n'}}(t))$ are either $F_0(\varrho^{-1}S_1(t)-)$ or $F_0(\varrho^{-1}S_1(t))$. Combining this argument with Lemmas~\ref{lemma:est-2},~\ref{lemma:jumps},~\ref{lemma:increase}, and the assumption that $F_0$ is non-decreasing, we get: 
$$
F_0(\varrho^{-1}S_1(t)-) \leq m_0(S_0(t))\quad \mbox{ and }\ 
F_0(\varrho^{-1}S_1(t)) \geq m_0(S_0(t)),\quad t \geq 0.
$$
Here we used the fact that $T_{A_n}(t) \leq t$ and $T_{B_n}(t) \leq t$, thus 
\begin{align*}
 0\leq    \sup_{s \in[0,T_{A_n}(t)]}(x_n(s-)-x_n(s)) & \le \sup_{s \in[0,t]}(x_n(s-)-x_n(s)) \to0,\quad n\to\infty,
\\
0\leq   \sup_{0\leq s_1\leq s_2\leq T_{B_n}(t)}(F_n(s_1)-F_n(s_2)) & \le \sup_{0\leq s_1\leq s_2\leq t}(F_n(s_1)-F_n(s_2))  \to0,\quad n\to\infty.
\end{align*}
The inequality 
\begin{equation}
\label{eq:find_delay}
 F_0(\varrho^{-1} S_1(t)-)\leq m_0(S_0(t)) \leq  F_0(\varrho^{-1}S_1(t)),\quad t \geq 0.
\end{equation}
implies $\varrho^{-1} S_1(t)=F_0^{-1}( m_0(S_0(t)))$. Since $S_0(t)+S_1(t)=t$, we can rewrite this as 
\[\varrho^{-1}(t-S_0(t))=F_0^{-1}(m_0(S_0(t))).\]
Therefore, 
\begin{equation}
\label{eq:S-0}
S_0(t)+\varrho F_0^{-1}(m_0(S_0(t)))=t,\quad t \ge 0.
\end{equation}
By Remark~\ref{rmk:m-0}, the function $A(s):=A_\varrho(s):= s+\varrho F_0^{-1}(m_0(s))$ is continuous, strictly increasing, $A(0)=0$, and $A(\infty)=\infty$. 
Therefore,   
$S_0(t)=A^{-1}(t)$ is continuous and strictly increasing. Combining the above formulas, we get:
\[
 S_1(t)=\varrho F_0^{-1}( m_0(S_0(t)))= \varrho F_0^{-1}( m_0(A^{-1}(t))).
\]

{\it Step 3.} It follows from  Steps 1 and 2 that there exists a sub-sequence $(n')$ such that we have $T_{A_{n'}} \to A^{-1}$ in $\mathcal C$. Similarly, we can show that every sub-sequence $(\tilde{n})$ have their own sub-sub-sequence $(\tilde{n}')$ such that $T_{A_{\tilde{n}'}} \to A^{-1}$ in $\mathcal C$. Since the space $\mathcal C$ is metric, this implies that $T_{A_n} \to A^{-1}$ in $\mathcal C$.
Since $ T_{B_{n}}(t)=t-T_{A_{n}}(t) $ we have convergence $ T_{B_{n}}\to S_1=\varrho F_0^{-1}( m_0(A^{-1}))$ too. It follows from \cite[Theorem 13.2.2]{Whitt} that 
\begin{equation}
\label{eq:conv_x0TA}
x_n\circ T_{A_n}\to x_0\circ A^{-1}\ \mbox{in}\   \cD 
\end{equation}
because $A^{-1}$ is continuous and strictly increasing. 

{\it Step 4.}  For brevity, we give the proof only for $\rho_n=1$. In the general case nothing will change, but the presence of $\rho_n$ in the numerators or denominators obscures the idea of the proof. Let us verify  the   convergence:
\begin{equation}
\label{eq:345}
F_n(T_{B_{n}})\to F_0(S_1) \ \ \mbox{in }\ \ \cD.
\end{equation}
Apply Lemma \ref{prop:criterion}. Assume that $t_n\to t_0, n\to\infty.$ Then $\wt t_n:=T_{B_{n}}(t_n)\to S_1(t_0)=: \wt t_0 $ because we have convergence $ T_{B_{n}}\to S_1$ in $\cC.$ Hence, by Lemma~\ref{prop:criterion}, the limit points of $(F_n(T_{B_{n}}(t_n)))=(F_n(\wt t_n))$ may be only $F_0(\wt t_0)= F_0(S_1(t_0))$ or $F_0(\wt t_0-)= F_0(S_1(t_0)-)$. 
Note that the function $S_1$ is only non-decreasing, so it is possible that $F_0(S_1(t_0)-)\neq F_0\circ S_1(t_0-).$ Conditions of Lemma~\ref{prop:criterion} are satisfied for $t_0$ if, for example, $\wt t_0=S_1(t_0)$ is a point of continuity of $F_0,$ or $\lim_{t\to t_0} F_n(T_{B_{n}}(t)) = F_0(\wt t_0)= F_0(S_1(t_0))$. Hence, the non-trivial case is only the one when $\wt t_0=S_1(t_0)=F_0^{-1}(m_0(A^{-1}(t_0)))$ is a point of jump of $F_0.$ Let $\alpha:=F_0(\wt t_0-)$ and $\beta:=F_0(\wt t_0).$ Then $\alpha<\beta.$ Notice that 
$F_0^{-1}(z)<\wt t_0$   for $z<\alpha,$  $F_0^{-1}(z)>\wt t_0$   for $z>\beta,$  and moreover $\alpha\leq m_0(A^{-1}(t_0))\leq \beta$. Assume at first that $m_0(A^{-1}(t_0))=\alpha.$ It follows from Theorem assumptions that $A^{-1}(t_0)$ is a point of growth of $m_0$. Thus $m_0(A^{-1}(s))<\alpha$ for $s<t_0$, because $A$ is strictly increasing. Hence   
\[
 F_0(S_1(t_0)-) = F_0(F_0^{-1}(m_0(A^{-1}(t_0)))-) =  F_0(F_0^{-1}(m_0(A^{-1}(t_0-))) )=F_0\circ S_1(t_0-). 
\]
Conditions \ref{prop:criterion:2} and \ref{prop:criterion:3} of Lemma \ref{prop:criterion} hold for the point $t_0$ because $F_0$ is strictly increasing. It remains to consider the case when $m_0(A^{-1}(t_0))\in(\alpha, \beta].$   Lemmas~\ref{lemma:est-2},~\ref{lemma:jumps},~\ref{lemma:increase}  and the assumption that $F_0$ is non-decreasing imply
\[
 \varliminf_{t\to t_0} F_n(T_{B_{n}}(t))\geq \varliminf_{t\to t_0} m_n(T_{A_{n}}(t)) = m_0(A^{-1}(t_0))>\alpha.
\]
Above, the limit points of $( F_n(T_{B_{n}}(t)))$ may be only $\alpha$ or  $\beta.$ So $\lim_{t\to t_0} F_n(T_{B_{n}}(t))$
exists and is equal to $\beta.$ Thus all conditions of 
 Lemma \ref{prop:criterion} holds and this completes the proof of \eqref{eq:345}. 

{\it Step 5.} From Definition~\ref{defn:switch}, $y_n(t)=x_n(T_{A_n}(t))+ F_n(T_{B_n}(t))$. 
Recall from~\eqref{eq:conv_x0TA} and~\eqref{eq:345}:
$$
x_n(T_{A_n})\to x_0(A^{-1})\quad \mbox{ and } \quad F_n(T_{B_n} )\to F_0(F_0^{-1}(m_0(A^{-1})))\quad \mbox{in}\quad \mathcal D.
$$
Convergence $y_n\to y_0$ in $\cD$, where 
\[
y_0(t)=x_0(A^{-1}(t))+ F_0(F_0^{-1}(m_0(A^{-1}(t)))),\ t\geq0,
\]
is proved if we show that functions $x_0(A^{-1})$ and  $F_0(F_0^{-1}(m_0(A^{-1})))$ do not have joint points of discontinuity, see Lemma \ref{lemma:sum}. The function  $A^{-1}$ is continuous. Thus, it suffices to show
\bel{eq:absence of discont}
 x_0 \mbox{ and  } F_0(F_0^{-1}(m_0)) \mbox{ do not have joint points of discontinuity. }
 \ee
Assume that the function $F_0(F_0^{-1}(m_0))$ is discontinuous at a point $t$, i.e., $F_0(F_0^{-1}(m_0))$ has a positive  jump at $t$. Set $s:=F_0^{-1}(m_0(t))$. Then $F$ has a jump at $s$.
Denote $\alpha:=F_0(s-)< F_0(s) =:\beta.$ 
Notice that $F_0^{-1}(z)=s$ for $ z\in[\alpha, \beta]$, 
$F_0^{-1}(z)<s$ for $ z<\alpha$, and $F_0^{-1}(z)>s$ for $ z>\beta$. Since $m_0$ is continuous and non-decreasing we must have $m_0(t)=\alpha=F_0(s-)$
and $m_0(\tilde t)<\alpha$ for $\tilde t<t.$ It follows from assumption (e)
 of the Theorem that $t$ is a point of growth of $m_0.$ Hence $x_0$ can't have positive jump at $t$. Since $x_0$ does not have negative jumps at all by assumptions of the theorem, the function $x_0$ is continuous at $t.$  This proves \eqref{eq:absence of discont} and hence the theorem in the case $\rho\in(0,\infty)$.

\subsection{Proof of   Theorem \ref{thm:main} for $\varrho=0$.} Similarly to the case $\varrho\in(0,\infty)$, by Lemma \ref{lemma:est-1} 
\[
\varlimsup_{n\to\infty}[-m_{n}(T_{A_{n}}(t))+F_{n}(T_{B_{n}}(t)/\rho_n-))] \leq 0.
\]
Since $T_{A_n}(t)+T_{B_n}(t)=t$ and $F_0(\infty)=\infty$, we have convergence $T_{A_n}(t) \to t$ and $T_{B_n}(t) \to 0$ for any fixed $t\geq 0$. Moreover, since all functions are non-decreasing in $t$ the convergence is locally uniform. It follows from \cite[Theorem 13.2.2]{Whitt} that 
 $x_{n}(T_{A_{n}})\to x_0$  and  $m_{n}(T_{A_{n}})\to m_0$ in $\cD$
 as $n\to\infty$. Moreover, continuity of $m_0$ implies the locally uniform convergence $m_{n}(T_{A_{n}})\to m_0$. Thus
\begin{equation}
\label{eq:lower}
-m_0(t)+\varlimsup_{n\to\infty} F_n(T_{B_n}(t)/\varrho_n-)  \leq 0.
\end{equation}
Using Lemma \ref{lemma:est-2}, similarly to the reasoning above, we obtain the inequality:
\begin{equation}
\label{eq:upper}
-m_0(t)+\varliminf_{n\to\infty} F_n(T_{B_n}(t)/\varrho_n)  \geq 0.
\end{equation}
Fix a $t > 0$ and set $s_n:=T_{B_n}(t)/\varrho_n$. Then we can rewrite~\eqref{eq:lower} and~\eqref{eq:upper} as 
\[
\varlimsup_{n\to\infty} F_n(s_n-)\leq m_0(t)\leq\varliminf_{n\to\infty} F_n(s_n).
\] 
Let $s_0$ be a limit point of $(s_n)$ (including infinity). By Lemma \ref{prop:criterion}, using the fact that $F_0$ is strictly increasing, we get the inequality
\begin{equation}
\label{eq:new-two-sided}
F_0(s_0-)\leq m_0(t)\leq F_0(s_0).
\end{equation}
Thus $s_0 := F_0^{-1}(m_0(t))$. Since any limit point of $(s_n)$ is determined uniquely we have convergence  $T_{B_n}(t)/\varrho_n =s_n\to s_0= F_0^{-1}(m_0(t))$ as $n \to \infty$ for any $t\ge0$, and hence a locally uniform convergence because all functions are non-decreasing and the limit is continuous. The same arguments as for the case $\varrho\in(0,\infty)$ imply convergence in $\cD$: $F_n(T_{B_n}/\varrho_n)\to F_0(F_0^{-1}(m_0))$, and finally 
$$
y_n:=x_n(T_{A_n})+ F_n(T_{B_n}/\varrho_n) \to y_0 := x_0 + F_0(F^{-1}_0(m_0)) = \mathcal S(x_0, F_0).
$$

\subsection{Proof of   Theorem \ref{thm:main} for $\varrho=\infty$.} 
Note that $\mathcal S(x_0, F_0)(t\wedge \sigma)=x_0(t\wedge \sigma).$ Since $\lim_{n\to\infty}\varrho_n=\infty$ and $T_{B_n}(t)\leq t$ we have the local uniform convergence 
$T_{B_n}(t)/\varrho_n\to0$ as $n\to\infty.$ Recall that $F_0(0)=0$ and $F_n\to F_0$ in $\cD.$ Due to Lemma \ref{prop:criterion},  for any $t\geq 0$ we have convergence $\lim_{n\to\infty}F_n (T_{B_n}(t)/\varrho_n)=F_0(0)=0$ and even locally uniform convergence because  functions $F_n (T_{B_n})$ are non-decreasing. It is clear that for any $t<\sigma$ we have $T_{A_n}(t)=t$ for sufficiently large $n.$ On the other hand, similarly to the case $\varrho\in(0,\infty)$,  Lemma \ref{lemma:est-2} implies:
\[
0\leq \varliminf_{n\to\infty} (F_n(T_{B_n}(t)/\varrho_n)-m_n(T_{A_n}(t)))= \varliminf_{n\to\infty}  (-m_n(T_{A_n}(t)))=
-\varlimsup_{n\to\infty}  m_n(T_{A_n}(t))\leq 0
\]
 for any $t\geq 0$. Since 
 $\inf\{s\geq 0\ | \ x_0(s)=0\}=\inf\{s\geq 0\ | \ x_0(s)<0\}$ by the assumption and $\inf\{s\geq 0\ | \ x_0(s)<0\}=\inf\{s\geq 0\ | \ m_0(s)>0\}$, we have  $\varlimsup_{n\to\infty}T_{A_n}(t)\leq \sigma$ for any $t$. We may conclude that
$T_{A_n}(t) \to t\wedge \sigma$. To prove the theorem, it suffices to verify convergence
$x_n( T_{A_n} )\to x_0(\cdot\wedge \sigma).$ Let us apply  Lemma \ref{prop:criterion}. If   $t_0\neq \sigma,$  then conditions of 
Lemma \ref{prop:criterion} are obviously satisfied.

Recall that  $\inf\{s\geq 0\ | \ x_0(s)=0\}=\inf\{s\geq 0\ | \ x_0(s)<0\}$. Hence  $x_0$  doesn't have a (positive) jump at $\sigma=\inf\{s\geq 0\ | \ x_0(s)=0\}$ and so $x_0$ is continuous at $t_0=\sigma$. Thus $x_0(\cdot\wedge \sigma)$ is continuous at $t_0= \sigma$, and  consequently conditions of Lemma \ref{prop:criterion} are satisfied for $t_0=\sigma$ too. $\square$

\section{Appendix}

\subsection{Proof of Lemma~\ref{lemma:existence}} The proof of existence is straightforward and is done exactly as in the classic Skorokhod problem. 
Let us show uniqueness: take two solutions $y_1, y_2$, and the corresponding boundary terms $l_1, l_2$. Then 
\begin{align*}
(y_1(t)-y_2(t))^2 &= 2\int_{(0,t]}(y_1(s)-y_2(s))\, \mathrm{d}(F(l_1(s))-F(l_2(s))) \\ & -\sum_{0<s\leq t}\Big(\Delta\big(F(l_1(s))-F(l_2(s))\big)\Big)^2 \\ & =
2\int_{(0,t]}(y_1(s)-y_2(s))\, \mathrm{d} F(l_1(s)) 
-2\int_{(0,t]}(y_1(s)-y_2(s))\, \mathrm{d} F(l_2(s))
\\ & -\sum_{0<s\leq t}\Big(\Delta\big(F(l_1(s))-F(l_2(s))\big)\Big)^2 \\ & =
-2\int_{(0,t]} y_2(s)\, \mathrm{d} F(l_1(s)) 
-2\int_{(0,t]} y_1(s)\, \mathrm{d} F(l_2(s))
\\ &  -\sum_{0<s\leq t}\Big(\Delta\big(F(l_1(s))-F(l_2(s))\big)\Big)^2\leq 0.
\end{align*}
 So, $y_1(t)=y_2(t)$ and thus $F(l_1(t))=F(l_2(t))$ for all $t\geq 0.$ Since $F$ is strictly increasing, $l_1(t)=l_2(t). \quad \square$

\subsection{Proof of Lemma~\ref{lemma:est-1}}
Recall that $B = \cup_{k=1}^{\infty}[\rho_k, \tau_k)$, $A = \cup_{k=1}^{\infty}[\tau_{k-1}, \rho_k)$. We also know that $y(t)\leq 0$ for $t\in B$, and therefore
\begin{equation}
    \label{eq:y-negative}
    y(t-)\leq 0\ \mbox{for}\ t\in  \bigcup\limits_{k=1}^{\infty}(\rho_k, \tau_k].
\end{equation} 

{\it Case 1.}  Assume  $t\in (\rho_k, \tau_k]$ for some $k, $ then 
$$
F(T_{B}(t)-)=\lim_{z\uparrow T_{B}(t)}F(z)= \lim_{s\uparrow t }F(T_{B} (s)) = F(T_{B} (t-)) \equiv (F\circ T_{B}) (t-).
$$
The second equality is true because $T_{B}$ is strictly increasing on $(\rho_k, \tau_k]$; thus, the left limit of the function $F$ at the point $T_B(\tau_{k-1})$ is equal to the left limit of the composition $F\circ T_B$ at the point $\tau_{k-1}$. Next, apply the formulas~\eqref{eq:main_eq_def_system}, ~\eqref{eq:m-n-sup},~\eqref{eq:y-negative}. We get:
\[
F(T_{B}(t)-)
= F(T_{B} (t-))= y(t-) - x(T_{A}(t-)) \leq 0 + m(T_{A}(t-))\leq  m(T_{A}(t)).
\]
For the last inequality, we use that $m(T_{A})$ is non-decreasing. This completes the proof of Lemma~\ref{lemma:est-1} in Case 1. 

{\it Case 2.}  Assume  $t\in (\tau_{k-1}, \rho_k]$ for some $k$.
On this interval, the function $T_{B}$ is constant. Therefore, $T_{B}(t)=T_{B}(\tau_{k-1})$. Thus
\begin{equation}
    \label{eq:prev-estimate-0}
F(T_{B}(t)-) = F(T_{B}(\tau_{k-1})-)= F(T_{B}(\tau_{k-1}-) ) \equiv (F\circ T_B)(\tau_{k-1}-).
\end{equation}
The second equality in~\eqref{eq:prev-estimate-0} from the fact that $T_{B}$ is strictly increasing  on $(\rho_{k-1}, \tau_{k-1}]$. Thus, the left limit of the function $F$ at the point $T_B(\tau_{k-1})$ is equal to the left limit of the composition $F\circ T_B$ at the point $\tau_{k-1}$.
The functions  $s \mapsto T_A(s-)$ and  $m$ are always non-decreasing. Thus $s \mapsto m(T_A(s-))$ is also non-decreasing. Using \eqref{eq:main_eq_def_system} again, we get
\begin{equation}
\label{eq:prev-estimate-1}
F(T_{B}(\tau_{k-1}-)) =
y(\tau_{k-1}- ) - x(T_{A}(\tau_{k-1}- )).
\end{equation}
Next, using~\eqref{eq:m-n-sup} and \eqref{eq:y-negative}, we get:
\begin{equation}
\label{eq:final-estimate}
y(\tau_{k-1}- ) - x(T_{A}(\tau_{k-1}- )) \le 0 - x(T_{A}(\tau_{k-1}))\leq  m(T_{A}(\tau_{k-1}))  \leq m(T_{A}(t)).
\end{equation}
Combining~\eqref{eq:prev-estimate-0}, ~\eqref{eq:prev-estimate-1},~\eqref{eq:final-estimate}, we complete the proof of Lemma~\ref{lemma:est-1} in Case 2. 

\subsection{Proof of Lemma~\ref{lemma:est-2}} We prove the statement by induction. First, the induction base: $[0, \rho_1)$, we are in regime $A$. Thus $T_A(t) = t$ and $T_B(t) = 0$. Therefore, $x(s) = y(s)$ for $s \in [0, \rho_1)$. Hence we have:
$$
\rho_1 = \inf\{s \ge 0\mid y(s) \leq -\delta\} = \inf\{s \ge 0\mid x(s) \leq -\delta\}.  
$$
Thus $m(t) \leq \delta$ for $t \in [0, \rho_1)$. Note that $F(0)= 0$, and the two suprema in the right-hand side of the inequality of Lemma~\ref{lemma:est-2} are non-negative.  This proves the lemma statement on $[0, \rho_1)$.

Before the induction step, notice that regimes can switch only when $x$ attains its minimum or $F$ attains its maximum:
\[
x(T_A(\rho_k))=-m(T_A(\rho_k)),\ \ F(T_B(\tau_k))=\max_{s\in [0,\tau_k]}F(T_B(s)), \ \ k\geq 1.
\]

{\it Case 1.} Assume the statement is true on $[0,\rho_k)$. Let us show it for $t\in [\rho_k, \tau_k)$. We have
\begin{align}
\label{eq:primary}
\begin{split}
F(T_B(t)) = \big(F(T_B(t))-F(T_B(\rho_k))\big)+
\big(F(T_B(\rho_k))-F(T_B(\rho_k-))\big) +\\
\big(F(T_B(\rho_k-)) + x(T_A(\rho_k-))\big)
+\big(x(T_A(\rho_k))-x(T_A(\rho_k-))\big)- x(T_A(\rho_k)).
\end{split}
\end{align}
Note that $T_A$ is constant on $[\rho_k, \tau_k)$: $T_A(t) = T_A(\rho_k)$. Thus $x(T_A(t))=x(T_A(\rho_k))$, and
\begin{equation}
\label{eq:T-A-const}
x(T_A(t)) = x(T_A(\rho_k))=-m(T_A(\rho_k)) = -m(T_A(t)).
\end{equation}
Similarly, $T_B$ does not grow on $[\tau_{k-1}, \rho_k]$, thus we have equality $F(T_B(\rho_k))=F(T_B(\rho_k-))$.
Combining~\eqref{eq:primary}, ~\eqref{eq:T-A-const}, we get:
\begin{align}
\label{eq:secondary}
\begin{split}
F(T_B(t)) & \geq  \inf_{0\leq s_1\leq s_2\leq t} \big(F(T_B(s_2))-F(T_B(s_1))\big) + 0 + y(\rho_k-)  \\ & + \big(x(T_A(\rho_k)) -x(T_A(\rho_k-))\big)+ m(T_A(t)).
\end{split}
\end{align}
Recalling that $y(\rho_k-)\geq -\delta$ and combining~\eqref{eq:secondary},~\eqref{eq:T-A-const} wit the elementary estimates:
\[
\inf_{0\leq s_1\leq s_2\leq t }\big(F(T_B(s_2))-F(T_B(s_1))\big)  {=}  -\sup_{0\leq t_1\leq t_2\leq T_B(t)}\big(F(t_1)-F(t_2)\big);
\]
\[
x(T_A(\rho_k))-x(T_A(\rho_k-))   \ge -\sup_{0\leq s \leq t } \big( x(T_A(s-))-x(T_A(s))\big) = -\sup_{0\leq z \leq T_A(t) } \big( x(z-)-x(z)\big),
\]
 we complete the proof of Lemma~\ref{lemma:est-2} in Case 1. 

{\it Case 2.} Assume the statement on $[0,\tau_{k-1})$, and show it for $t\in [ \tau_{k-1}, \rho_k)$. Use~\eqref{eq:main_eq_def_system} to get
\[
F(T_B(t)) = y(t) - x(T_A(t))\geq -\delta - x(T_A(t))\geq  -\delta +m(T_A(t)). 
\]
This completes the proof of Lemma~\ref{lemma:est-2} in Case 2.

\subsection{Proof of Lemma~\ref{lemma:jumps}} Define the function $j : \mathcal D_T \to \mathbb R$ as the largest negative jump of a function $x \in \mathcal D_T$. This function is well-defined since any function $x \in \mathcal D_T$ has only countably many jumps and only finitely many jumps which exceed any given positive level. The function $j$ is continuous in the Skorokhod topology of $\mathcal D_T$. This can be proved similarly to \cite[Example 12.1]{Billingsley}, which shows that the largest jump (regardless of direction) of a function in $\mathcal D$ is a continuous function $\mathcal D \to \mathbb R$. Since $x_0$ has no negative jumps, obviously $j(x_0) = 0$. Therefore, $\sup_{s \in [0, T]}(x_n(s-) - x_n(s)) = j(x_n)(T)\to 0. \quad \square$

\subsection{Proof of Lemma~\ref{lemma:increase}} Assume that $F_0$ is continuous at $T$. There exists for every $n$ a strictly increasing one-to-one function $\lambda_n : [0, T] \to [0, T]$ such that $F_n(\lambda_n) \to F_0$ uniformly on $[0, T]$. We can rewrite
\begin{equation}
\label{eq:switch}
\sup\limits_{0 \leq s_1 \leq s_2 \leq T}(F_n(s_1) - F_n(s_2)) = \sup\limits_{0 \leq s_1 \leq s_2 \leq T}(F_n(\lambda_n(s_1)) - F_n(\lambda_n(s_2))).
\end{equation}
Next, by uniform convergence $F_n(\lambda_n) \to F_0$:
\begin{equation}
\label{eq:convergence}
\sup\limits_{0 \leq s_1 \leq s_2 \leq T}(F_n(\lambda_n(s_1)) - F_n(\lambda_n(s_2))) \to \sup\limits_{0 \leq s_1 \leq s_2 \leq T}(F_0(s_1) - F_0(s_2)).
\end{equation}
Since the function $F_0$ is non-decreasing, the right-hand side of~\eqref{eq:convergence} is zero. Together with~\eqref{eq:switch}, this completes the proof of Lemma~\ref{lemma:increase}. $\square$

\end{document}